\documentclass[11pt]{amsart}
\usepackage[numbers, square]{natbib}
\usepackage{amssymb}
\usepackage{amsmath}
\usepackage{amsfonts}
\usepackage{geometry}
\usepackage{graphicx}
\usepackage{amsthm}
\usepackage{hyperref}
\usepackage{wrapfig}

\usepackage[dvipsnames]{xcolor}

\providecommand{\U}[1]{\protect\rule{.1in}{.1in}}
\theoremstyle{plain}

\newtheorem{corollary}{Corollary}

\newtheorem{lemma}{Lemma}

\newtheorem{proposition}{Proposition}
\newtheorem{remark}{Remark}

\newtheorem{theorem}{Theorem}
\newtheorem*{theorem*}{Theorem}
\numberwithin{equation}{section}

\newcommand{\disp}{\displaystyle}

\newcommand{\eps}{\varepsilon}
\newcommand{\vp}{\varphi}

\newcommand{\al}{\alpha}
\newcommand{\be}{\beta}

\newcommand{\te}{\theta}

\newcommand{\om}{\omega}
\newcommand{\Om}{\Omega}
\newcommand{\si}{\sigma}

\newcommand{\iny}{\infty}

\newcommand{\su}{\subset}
\newcommand{\LP}{\Delta}
\newcommand{\gr}{\nabla}

\newcommand{\norm}[1]{\left\| #1\right\|}
\newcommand{\innp}[1]{\left< #1 \right>}
\newcommand{\abs}[1]{\left\vert#1\right\vert}
\newcommand{\set}[1]{\left\{#1\right\}}
\newcommand{\brac}[1]{\left[#1\right]}
\newcommand{\pr}[1]{\left( #1 \right) }
\newcommand{\pb}[1]{\left( #1 \right] }

\newcommand{\R}{\ensuremath{\mathbb{R}}}

\newcommand{\C}{\ensuremath{\mathbb{C}}}

\begin{document}
\title{Quantitative unique continuation for Schr\"odinger operators}

\author{ Blair Davey}
\address{
Department of Mathematics\\
The City College of New York\\
New York, NY 10031, USA\\
Email: bdavey@ccny.cuny.edu }
\thanks{\noindent{Davey is supported in part by the Simons Foundation Grant number 430198.}}
\subjclass[2010]{35J15, 35J10, 35A02.}
\keywords {Carleman estimates, unique continuation, singular lower order terms, vanishing order}

\begin{abstract}
We investigate the quantitative unique continuation properties of solutions to second order elliptic equations with singular lower order terms.
The main theorem presents a quantification of the strong unique continuation property for $\LP + V$. 
That is, for any non-trivial $u$ that solves $\LP u + V u = 0$ in some open, connected subset of $\R^n$, we estimate the vanishing order of solutions in terms of the $L^t$-norm of $V$.
Our results apply to all $t > \frac n 2$ and $n \ge 3$. 
With these maximal order of vanishing estimates, we employ a scaling argument to produce quantitative unique continuation at infinity estimates for global solutions to $\LP u + V u = 0$.
To handle $V \in L^t$ for every $t \in \pb{\frac n 2, \iny}$, we prove a novel $L^p - L^q$ Carleman estimate by interpolating a known $L^p - L^2$ estimate with a new endpoint Carleman estimate.
This new Carleman estimate may also be used to establish improved order of vanishing estimates for equations with a first order term, those of the form $\LP u + W \cdot \gr u + V u = 0$.
\end{abstract}

\maketitle

\section{Introduction}

In this article, we study the quantitative unique continuation properties of solutions to elliptic equations with singular lower order terms.
More specifically, we establish bounds for the {\em order of vanishing} and the {\em optimal rate of decay at infinity} of solutions to equations of the form $\LP u + V u = 0$, where $V \in L^t$ for any $t \in \pb{\frac n 2, \iny}$ and $n \ge 3$.
The order of vanishing results may be interpreted as a quantification of the strong unique continuation property.
The estimates for the rate of decay at infinity follow from the order of vanishing estimates via the scaling argument of Bourgain and Kenig \cite{BK05}, but also relate to Landis' conjecture.
We consider elliptic equations with a non-trivial first order term as well.

Recall that a partial differential operator $L$ is said to have the {\em unique continuation property} (UCP) if whenever $u$ is a solution to $L u = 0$ in some domain $\Om$ and $u \equiv 0$ on an open subset of $\Om$, then it necessarily follows that $u \equiv 0$ throughout $\Om$ as well.
Further, we say that $L$ has the {\em strong unique continuation property} (SUCP) if whenever $u$ is a solution to $L u = 0$ in $\Om$ and there exists $x_0 \in \Om$ at which $u$ vanishes to infinite order, then $u \equiv 0$ in $\Om$.
Depending on the underlying function space, there are different ways to interpret what it means for a solution to vanish to infinite order at a point.
The study of (strong) unique continuation for the operator $\LP + V$, where $V \in L^t$, has an extensive history.
The most notable result is the theorem of Jerison and Kenig \cite{JK85} in which they establish that the strong unique continuation property holds for $V \in L^{n/2}$ for every $n \ge 3$.

Here we consider solutions to equations of the form $\LP u + V u = 0$ in $\Om \su \R^n$, where $n \ge 3$ and $V \in L^t$ for any $t \in \pb{\frac n 2, \iny}$.
By the result of Jerison and Kenig \cite{JK85}, we know that all such solutions satisfy the SUCP, and therefore, only trivial solutions may vanish to infinite order at a point in the domain.
As such, we want to quantify the SUCP by finding the maximal rate at which a non-trivial solution can vanish.
We formulate this as follows: If $u$ is a bounded, normalized solution to $\LP u + V u = 0$, we seek a lower bound of the form
$$\norm{u}_{L^\iny\pr{B_r}} \gtrsim r^{\be} \;\; \text{ as } \;\; r \to 0,$$
where $\be$ is some function that encapsulates information about the operator, $\LP + V$.
Our Theorem \ref{OofV} below gives this quantification.

While order of vanishing estimates provide local information about unique continuation, these estimates, in combination with the scaling argument from \cite{BK05}, may be used to establish global unique continuation results.
In fact, this technique has been extensively used in the study of Landis' conjecture (see, for example, \cite{BK05}, \cite{Dav14}, \cite{LW14}, \cite{KSW15}, \cite{DKW17}, \cite{DW17}, \cite{DKW18}, \cite{Dav18u}).
Landis' conjecture states that if $u$ and $V$ are bounded functions that satisfy $\LP u + V u = 0$ in $\R^n$, and $\abs{u\pr{x}} \lesssim \exp\pr{- c \abs{x}^{1+}}$, then $u \equiv 0$.
In Theorem \ref{UCV}, we establish a quantitative Landis-type theorem for any $V \in L^t$, $t \in \pb{\frac n 2, \iny}$.

This article continues the work that was initiated in \cite{DZ17} and \cite{DZ18} in which we considered equations of the form $\LP u + W \cdot \gr u + V u = 0$, where both of the potential functions are assumed to be singular.
The results of \cite{DZ17} apply to $n \ge 3$, while \cite{DZ18} restricts to $n = 2$.
In this article, we present improvements to some of the results presented in \cite{DZ17}.
When $W \equiv 0$, the results of \cite{DZ17} only apply to $V \in L^t$ where $t > \frac{4 n^2}{7n+2}$.
To close the gap between $\frac n 2$ and $\frac{4 n^2}{7n+2}$, we prove a new Carleman estimate.
The approach is to interpolate between the $L^p - L^2$ Carleman estimate from \cite{DZ17} and a new $L^p - L^{\frac{2n}{n-2}}$ estimate that we establish here using new ideas.

Although the motivation behind this article was to close the gap between $\frac n 2$ and $\frac{4 n^2}{7n+2}$ for admissable $t$-values, the new Carleman estimate here also applies to operators with a first-order term.
As such, our Theorem \ref{OofVW} establishes order of vanishing estimates for solutions to equations of the form $\LP u + W \cdot \gr u + V u = 0$, where we assume that $V \in L^t$ and $W \in L^s$.
Compared to the results in \cite{DZ17}, this theorem improves upon the rate of decay in some cases, but doesn't increase the range of $s$- or $t$-values.
Moreover, the improvement is not sufficient enough to imply a better unique continuation at infinity result than the one presented in \cite{DZ17}.

Our main focus is on solutions to elliptic equations of the form
\begin{equation}
\LP u+V u=0.
\label{goal}
\end{equation}
As is standard, we use the notation $B_r\pr{x}$ to denote a ball of radius $r$ centered at the point $x$.
When the center is understood from the context, we simply write $B_r$.

A priori, we assume that $u \in W^{1,2}_{loc}\pr{B_R}$ is a weak solution to \eqref{goal} in $B_R \su \R^n$ for $n \ge 3$.
H\"older and Sobolev inequalities imply that there exists a $p \in \pb{\frac{2n}{n+2}, 2}$, depending on  $t$, such that $V u \in L^p_{loc}\pr{B_R}$.
By regularity theory, it follows that $u \in W_{loc}^{2,p}\pr{B_R}$ and therefore, $u$ is a solution to \eqref{goal} almost everywhere in $B_R$.
Moreover, by de Giorgi-Nash-Moser theory, we have that $u \in L^\iny_{loc}\pr{B_R}$.
Therefore, when we say that $u$ is a solution to \eqref{goal} in $B_R$, it is understood that $u$ belongs to $L^\iny_{loc}\pr{B_R} \cap W^{1,2}_{loc}\pr{B_R} \cap W^{2,p}_{loc}\pr{B_R}$ and $u$ satisfies equation \eqref{goal} almost everywhere in $B_R$.
For solutions to equations of the form $\LP u + W \cdot \gr u + V u = 0$, an analogous set of statements can be made, and we interpret the meaning of solution similarly.

The order of vanishing result for $\LP + V$ is as follows.

\begin{theorem}
Let $t \in \pb{\frac{n}{2}, \iny}$. 
Assume that for some $M \ge 1$, $\norm{V}_{L^t\pr{B_{R_0}}} \le M$. 
Let $u : B_{10} \to \C$ be a solution to \eqref{goal} in $B_{10}$ that is bounded and normalized in the sense that
\begin{align}
& \norm{u}_{L^\iny\pr{B_{6}}}\le \hat{C},
\label{bound} \\
& \norm{u}_{L^\iny\pr{B_{1}}}\geq 1.
\label{normal}
\end{align}
Then the maximal order of vanishing for $u$ in  $B_{1}$ is less than $C M^\mu$. 
That is, for any $\eps \in \pr{0, \eps_0}$, any $x_0\in B_1$, and any $r$ sufficiently small,
\begin{align*}
\norm{u}_{L^\iny\pr{B_{r}\pr{x_0}}} &\ge c r^{C M^\mu},
\end{align*}
where 
$C = C\pr{n, t, \hat C, \eps}$,
$\disp \mu = \left\{\begin{array}{ll}
\frac{4t}{6t - 3n + 2} &  t \in \pb{n, \iny} \\
\frac{4nt} {\pr{2t-n}\pr{3n+2}} + \tilde \eps & t \in \pb{\frac{n}{2}, n} 
\end{array}\right.$,
$\tilde \eps = C\pr{n, t, \eps} \eps$, $c=c\pr{n, t, \hat C, \eps}$.
\label{OofV}
\end{theorem}

\begin{figure}[h]
\begin{center}
\includegraphics[height=2in,width=4in,]{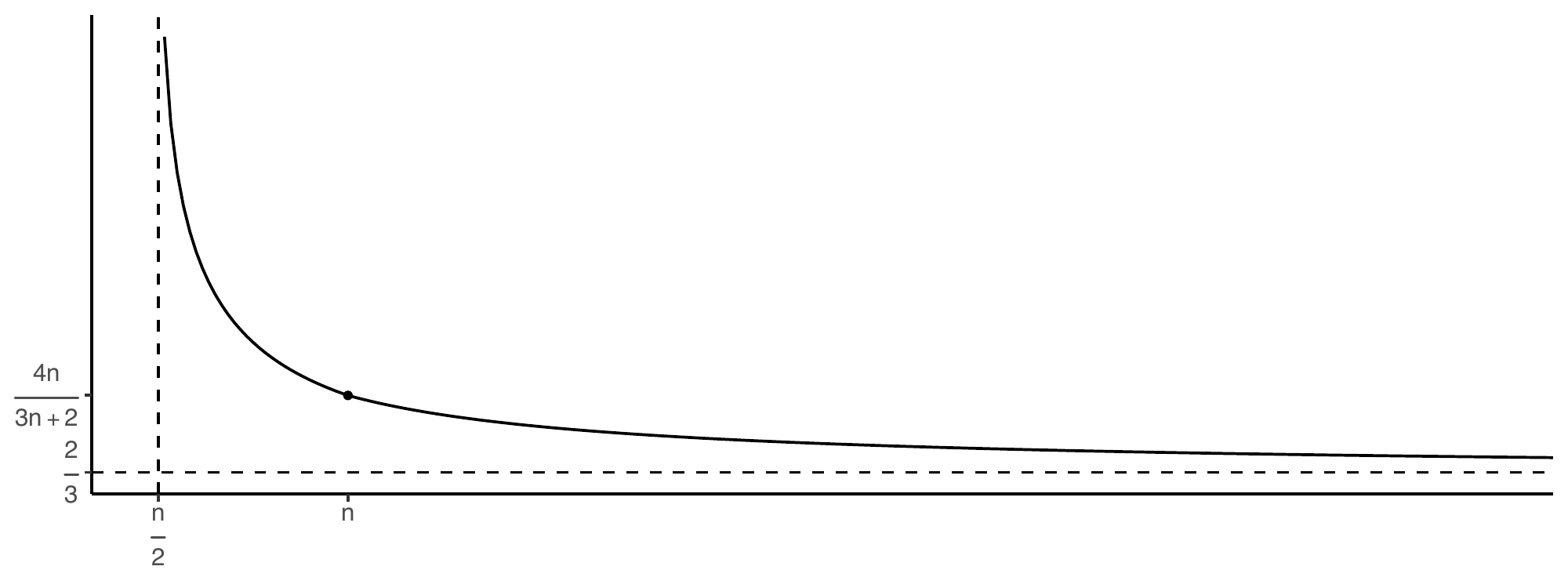}
\caption{The graph of $\mu$ as a function of $t$.
As expected, there is a vertical asymptote at $t = \frac n 2$ and a horizontal asymptote at $\frac 2 3$.}
\end{center}
\end{figure}

We use this order of vanishing estimate, in combination with the scaling argument from \cite{BK05}, to prove the following quantitative unique continuation at infinity theorem for solutions.
This result may be interpreted as a quantitative Landis-type theorem.

\begin{theorem}
Assume that $\norm{V}_{L^t\pr{\R^n}} \le A_0$ for some $t \in \pb{ \frac{n}{2}, \iny}$.
Let $u : \R^n \to \C$ be a solution to \eqref{goal} in $\R^n$ for which $\norm{u}_{L^\iny\pr{\R^n}} \le C_0$ and $\abs{u\pr{0}} \ge 1$.
Then for any $\eps \in \pr{0, \eps_0}$ and any $R$ sufficiently large,
\begin{equation*}
\inf_{\abs{x_0} = R} \norm{u}_{L^\iny\pr{B_1\pr{x_0}}} \ge \exp\pr{-C R^\Pi \log R},
\end{equation*}
where 
$\disp \Pi =  
\left\{\begin{array}{ll}
\frac{4\pr{2t- n}}{6t - 3n + 2} & t\in \pb{n, \iny} \\
\frac{4n} {3n+2} + \tilde \eps & t \in \pb{\frac{n}{2}, n} 
\end{array}\right.$,
$\tilde \eps = C\pr{n, t, \eps} \eps$, $C = C\pr{n, t, A_0, C_0, \eps}$.
\label{UCV}
\end{theorem}

\begin{figure}[h]
\begin{center}
\includegraphics[height=2in,width=4in,]{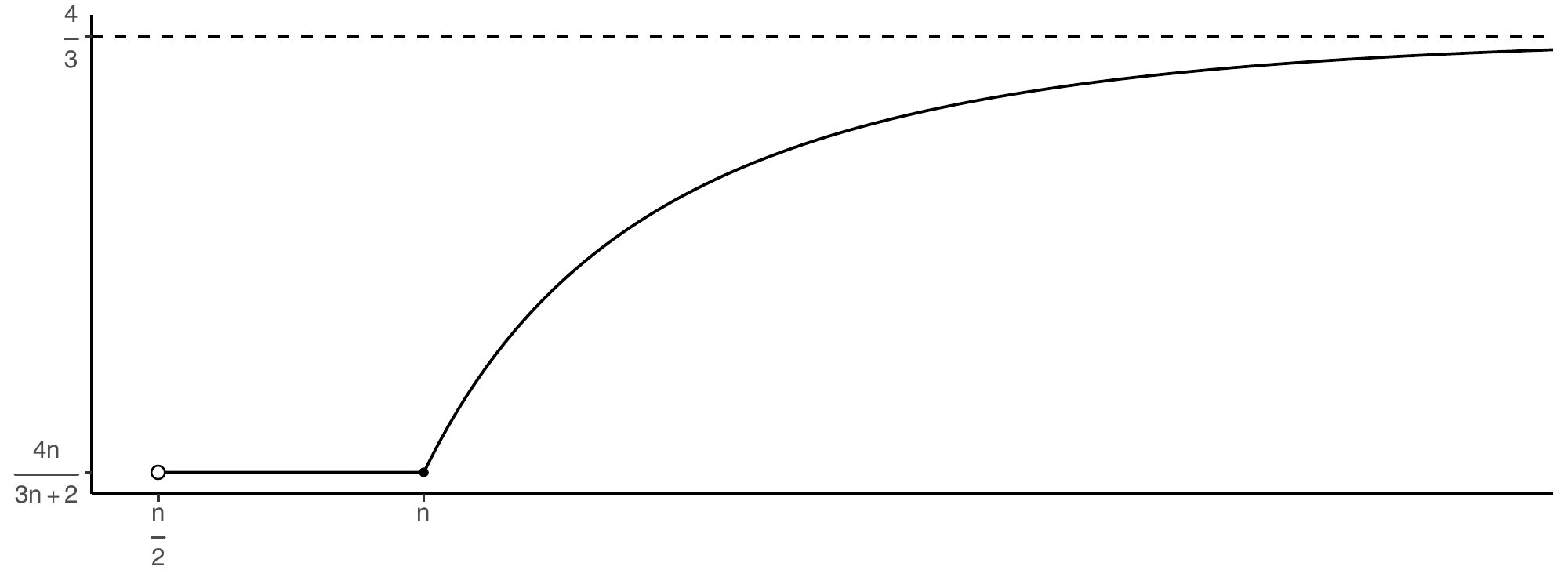}
\caption{The graph of $\Pi$ as a function of $t$.}
\end{center}
\end{figure}

\begin{remark}
For both of these theorems, we may define $\eps_0 = 2\pr{2 - \frac n t}$.
Moreover, if $t \in \pb{n, \iny}$, then all statements hold with $\eps = 0$.
\end{remark}

\begin{remark}
If we compare Theorems \ref{OofV} and \ref{UCV} with Theorems 3 and 4 from \cite{DZ18} (the analogous results for $n = 2$), we see that our new results are consistent.
Specifically, if we evaluate $\mu$ and $\Pi$ at $n = 2$, we recover the formulas that were derived in \cite{DZ18}.
\end{remark}

The proof of Theorem \ref{UCV} is brief, so we present it now.

\begin{proof}
Let $u$ be a solution to \eqref{goal} in $\R^n$.
Fix $x_0 \in \R^n$ and set $\abs{x_0} = R$.
Let $u_R(x) = u(x_0 + Rx)$ and $V_R\pr{x} = R^2 V\pr{x_0 + R x}$.
Then, $\disp \norm{V_R}_{L^t\pr{B_{10}}} \le A_0 R^{2 - \frac n t}$ and
\begin{align*}
 \LP u_R\pr{x} + V_R\pr{x} u_R\pr{x}
&= R^2 \LP u\pr{x_0 + R x} + R^2 V\pr{x_0 + R x}u\pr{x_0 + R x}
= 0.
\end{align*}
Therefore, $u_R$ is a solution to a scaled version of \eqref{goal} in $B_{10}$.
Clearly,
\begin{align*}
\norm{u_R}_{L^\iny\pr{B_{10}}}
&= \norm{u}_{L^\iny\pr{B_{10R}\pr{x_0}}} \le C_0.
\end{align*}
Note that for $\disp\widetilde{x_0} := -x_0/R$, $\disp| \widetilde{x_0}| = 1$ and $\abs{u_R\pr{\widetilde{x_0}}} = \abs{u(0)} \ge 1$ so that $\disp\norm{u_R}_{L^\iny(B_1)} \ge 1$.
Thus, if $R$ is sufficiently large, then we may apply Theorem \ref{OofV} to $u_R$ with $M = A_0 R^{2 - \frac n t}$, $\hat C = C_0$, and $r = 1/ R$ to get
\begin{align*}
\norm{u}_{L^\iny\pr{{B_{1}(x_0)}}}
&= \norm{u_R}_{L^\iny\pr{B_{1/R}}}
\ge c(1/R)^{^{C \pr{A_0 R^{2 - \frac n t} }^\mu}} 
= \exp\set{-\brac{C \pr{A_0 R^{2 - \frac n t} }^\mu} \log R + \log c}.
\end{align*}
Since $\Pi = \mu\pr{2 - \frac n t}$, then
\begin{align*}
\norm{u}_{L^\iny\pr{{B_{1}(x_0)}}}
\ge & \exp\brac{-\pr{C A_0^\mu - \log c} R^\Pi \log R},
\end{align*}
as required.
\end{proof}

In our final theorem, we establish is an order of vanishing estimate for equations with a drift term.
The only difference between this theorem and the result that we previously obtained in \cite[Theorem 1]{DZ17} is that the value of $\mu$ has improved for smaller values of $t$.

\begin{theorem}
Let $s \in \pb{\frac{3n-2}{2}, \iny}$ and $t \in \pb{ n\pr{\frac{3n-2}{5n-2}}, \iny}$.
Assume that for some $K, M \ge 1$, $\norm{W}_{L^s\pr{B_{10}}} \le K$ and $\norm{V}_{L^t\pr{B_{10}}} \le M$.
Let $u : B_{10} \to \C$ be a solution to $\LP u + W \cdot \gr u + Vu = 0$ in $B_{10}$.
Assume that $u$ is bounded and normalized in the sense of \eqref{bound} and \eqref{normal}.
Then the maximal order of vanishing for $u$ in $B_{1}$ is less than $ C_1 K^\kappa + C_2 M^\mu$.
That is, for any $x_0\in B_1$ and any $r$ sufficiently small,
\begin{align*}
\norm{u}_{L^\iny(B_{r}(x_0))} &\ge c r^{\pr{ C_1 K^\kappa + C_2 M^\mu}},
\end{align*}
where $C_1 = C_1\pr{n, s, t, \hat C}$, $C_2 = C_2\pr{n, s, t, \hat C}$,
$\disp \kappa = \left\{\begin{array}{ll}
\frac{4s}{2s - \pr{3n-2}} & t \in \pb{\frac{sn}{s+n}, \iny} \medskip \\ 
\frac{4t}{\pr{5 - \frac 2 n}t - \pr{3n-2}} & t \in \pb{n\pr{\frac{3n-2}{5n-2}},\frac{sn}{s+n}} 
\end{array}\right.$,
$\disp \mu = \left\{\begin{array}{ll}
\frac{4s}{6s - \pr{3n-2}} & t \in \brac{s, \iny} \medskip \\
\frac{4 s t}{6 s t + \pr{n+2}t -4ns + n\pr{n-2}\pr{1 - \frac t s}} & t \in \pr{\frac{sn}{s+n}, s} \medskip \\
\frac{2t}{2t -n} & t \in \pb{n\pr{\frac{3n-2}{5n-2}}, \frac{sn}{s+n}}
\end{array}\right.$,
$c = c\pr{n, s, t, \hat C}$.
\label{OofVW}
\end{theorem}

As in the proof of Theorem \ref{UCV}, a scaling argument may be used to pass from this order of vanishing estimate to a unique continuation at infinity theorem.
Since such an argument shows that $\Pi = \max\set{\kappa\pr{1 - \frac n s}, \mu \pr{2 - \frac n t}}$, the result of this argument leads to exactly \cite[Theorem 2]{DZ17}.
In other words, the order of vanishing estimate here is an improvement over \cite[Theorem 1]{DZ17},  but the improvement is not substantial enough to imply a stronger unique continuation at infinity theorem.
For the precise statement of this unique continuation at infinity theorem, we refer the reader to \cite{DZ17} .

The remainder of this article is devoted to the proofs of Theorems \ref{OofV} and \ref{OofVW}.
In the following section, Section \ref{CE}, we state the new $L^p - L^q$ Carleman estimate for $\LP$ that allows us to consider all values of $t > \frac n 2$.
We then use this Carleman estimate to establish related Carleman estimates for the operators $\LP + V$ and $\LP + W \cdot \gr + V$.
In Section \ref{OofVProof}, we describe the proofs of Theorems \ref{OofV} and \ref{OofVW}.
First we use our Carleman estimate for the operator to establish a three-ball inequality, then we use the three-ball inequality to prove a propagation of smallness result that leads to the stated estimate.
Finally, Section \ref{CarlProofs} presents the proof of our most important tool, the new $L^p - L^q$ Carleman estimate that is stated below in Theorem \ref{Carlpq}.

\section{Carleman estimates}
\label{CE}

Within this section, we state our new $L^p - L^q$ Carleman estimate for $\LP$.
We also present the version of this theorem that was used in \cite{DZ17}, and discuss the relationship between the two estimates.
Then we use our new Carleman estimate for $\LP$ to establish a Carleman estimate for $\LP + V$, where $V \in L^t$ for any $t > \frac n 2$.
The Carleman estimate for $\LP + V$ in \cite{DZ17} was only applicable to $t > \frac{4 n^2}{7n+2}$, so there are significant improvements here that we point out.
Finally, we present a new estimate for operators of the form $\LP + W \cdot \gr + V$.
This Carleman estimate is very similar to the one achieved in \cite{DZ17}, but the differences will be indicated.

Our new Carleman estimate for the Laplacian is as follows.

\begin{theorem}
Let $\frac{2n}{n+2} < p \le 2 \le q \le \frac{2n}{n-2}$. 
There exists a constant $C = C\pr{n,p,q}$ and a sufficiently small $R_0 < 1$, such that for any $u\in C^{\iny}_{0}\pr{B_{R_0}(x_0)\backslash\set{x_0} }$ and any $\tau > 1$, one has
\begin{align}
&\tau^{\be} \norm{\pr{\log r}^{-1} e^{-\tau \phi\pr{r}}u}_{L^q(r^{-n}dx)}
+ \tau^{\be_1} \norm{ \pr{\log r}^{-1} e^{-\tau \vp\pr{r}} r \gr u}_{L^2(r^{-n}dx)} \nonumber \\
&\le C \norm{ \pr{\log r} e^{-\tau \phi\pr{r}} r^2 \LP u}_{L^p(r^{-n} dx)} ,
\label{mainCar}
\end{align}
where $\be =  \frac 3 2  - \frac{3n-2}{4}\pr{\frac 1 p - \frac 1 2} - \brac{1-  \frac{n-2}{4}\pr{\frac 1 p - \frac 1 2}}n \pr{\frac 1 2 - \frac{1}{q}}$, $\be_1 =\frac{1}{2}- \frac{3n-2}{4}\pr{\frac 1 p - \frac 1 2}$, and $\phi(r)=\log r + \log \pr{\log r}^2$.
 \label{Carlpq}
\end{theorem}

This crucial estimate is proved below in Section \ref{CarlProofs}, but we'll describe the approach here.
We first modify the ideas from \cite{DZ17} to prove a new $L^p - L^q$ Carleman estimate for the Laplacian.
The big idea here is to iterate Young's convolution inequality instead of iterating Carleman estimates for the first order factors of the Laplacian.
Then with interpolate our new estimate at $q = \frac{2n}{n-2}$ with the estimate from \cite[Theorem 5]{DZ17} at $q = 2$.
The $L^p - L^q$ Carleman estimate from \cite{DZ17} is given by:

\begin{theorem}[Theorem 5 in \cite{DZ17}]
Let $\frac{2n}{n+2} < p \le 2 \le q \le \frac{2n}{n-2}$. 
There exists a constant $C = C\pr{n,p,q}$ and a sufficiently small $R_0 < 1$ such that for any $u\in C^{\iny}_{0}\pr{B_{R_0}(x_0)\backslash\set{x_0} }$ and any $\tau>1$, one has
\begin{align}
& \tau^{\be_0} \norm{ \pr{\log r}^{-1} e^{-\tau \vp\pr{r}}u}_{L^q(r^{-n}dx)}
+ \tau^{\be_1} \norm{ \pr{\log r}^{-1} e^{-\tau \vp\pr{r}} r \gr u}_{L^2(r^{-n}dx)} \nonumber \\
&\le C \norm{(\log r ) e^{-\tau \vp\pr{r}} r^2 \LP u}_{L^p(r^{-n} dx)} ,
\label{mainCarO}
\end{align}
where $\be_0 = \frac 3 2 - \frac{3n-2} 4 \pr{\frac 1 p - \frac 1 2 } - n\pr{\frac 1 2 - \frac 1 q}$, and $\be_1$ and $\phi\pr{r}$ is as in the previous theorem.
 \label{CarlpqO}
\end{theorem}

If $\be \ge \be_0$, then Theorem \ref{Carlpq} is a better estimate than Theorem \ref{CarlpqO}.
Since $\be \ge \be_0$ if and only if $\frac{n\pr{n-2}} 4 \pr{\frac 1 p - \frac 1 2} \pr{\frac 1 2 - \frac 1 q} \ge 0$, which always holds, then Theorem \ref{Carlpq} improves upon Theorem \ref{CarlpqO} in all cases.
Moreover, since $\be = \be_0$ when either $p = 2$ or $q = 2$, then applications of the Carleman estimates leads to the same results when we are in this setting.
An examination of the proofs in \cite{DZ17} shows that  for any $t > n$, the optimal value is obtained by choosing $q = 2$.
Therefore, there will be some cases where our results follow exactly from the proofs in \cite{DZ17}.

To establish a Carleman estimate for the operator $\LP + V$, where $V \in L^t$, the optimal relationship between $p$ and $q$ is $\frac 1 p - \frac 1 q = \frac 1 t$.
And to quantify the strong unique continuation property for $\LP + V$, we require that our estimates hold for all sufficiently large values of $\tau$.
For any $t \le \frac{4 n^2}{7n+2}$, all admissible values of $p$ and $q$ lead that $\be_0 \le 0$, explaining why Theorem \ref{CarlpqO} cannot be used for $t \in \pb{\frac n 2, \frac{4 n^2}{7n+2}}$.
However, for every $t > \frac n 2$, we may choose $p$ and $q$ within the acceptable range so that $\be > 0$.
This observation has two implications.
First, it means that there exists a range of $t$-values where Theorem \ref{Carlpq} leads to a better quantification than Theorem \ref{CarlpqO}.
Second, it implies that we may now consider all values of $t > \frac n 2$.
Our Carleman estimate for $\LP + V$ is as follows.

\begin{theorem}[cf. Theorem 8 in \cite{DZ17}]
Assume that for some $t \in \pb{\frac n 2, \iny}$ and $M \ge 1$, $\norm{V}_{L^t\pr{B_{R_0}}} \le M$.
Then for any $\eps \in \pr{0, \eps_0}$, there exist constants $C_0 = C_0\pr{n,t, \eps}$, $C_1 = C_1\pr{n,t, \eps}$, and sufficiently small $R_0 < 1$  such that for any $u\in C^{\iny}_{0}\pr{B_{R_0}(x_0)\setminus \set{x_0}}$ and any $\tau \ge C_1 M^{\mu}$, one has
\begin{align}
\tau^{\be} \norm{\pr{\log r}^{-1} e^{-\tau \phi\pr{r}}u}_{L^2(r^{-n}dx)}
&\le  C_0 \norm{ \pr{\log r} e^{-\tau \phi\pr{r}} r^2\pr{ \LP u  + V u}}_{L^p(r^{-n} dx)} , \label{main1}
\end{align}
where $\be = \frac {3n+2} {4n} + \frac{3n-2} {8n}\eps$, 
$\disp \mu = \left\{\begin{array}{ll}
\frac{4t}{6t - \pr{3n-2}} & t \in \pb{n, \iny} \\
\frac{4nt}{\pr{2t - n}\pr{3n+2}} + \tilde \eps & t \in \pb{\frac{n}{2}, n}
\end{array}\right.$,
$\disp p = \left\{\begin{array}{ll}
\frac{2t}{t+2} & t \in \pb{n, \iny}  \\
 \frac{2n}{n + 2 - \eps} & t \in \pb{\frac{n}{2}, n}
\end{array}\right.$,
and $\tilde \eps = C\pr{n, t, \eps}\eps$.
\label{CarlpqV}
\end{theorem}

\begin{proof}
If $t > n$, then this result follows from \cite[Theorem 8]{DZ17}.
Therefore, we assume that $t \in \pb{\frac n 2, n}$.
We choose $p$ and $q$ so that $\frac 1 p - \frac 1 q = \frac 1 t$ and then
\begin{align*}
\be &= \frac 3 2 -  \frac n t
+ \brac{\frac{n + 2} 4 + \frac{n\pr{n-2}}{ 4t}}\pr{\frac 1 p - \frac 1 2 } 
- \frac{n\pr{n-2}} 4 \pr{\frac 1 p - \frac 1 2 }^2.
\end{align*}
Since $\be$ increases with respect to $\frac 1 p - \frac 1 2$ until it reaches a maximum at $\frac 1 p - \frac 1 2 = \frac{\pr{n+2}t +  n \pr{{n-2}}}{2nt\pr{n-2}}$ and $\frac{\pr{n+2}t +  n \pr{{n-2}}}{2nt\pr{n-2}} \ge \frac 1 n$ for $t \le n$, then $\be$ is maximized when $p$ is arbitrarily close to, but greater than, $\frac{2n}{n+2}$. 
Therefore, as in the statement, let $p = \frac{2n}{n + 2 - \eps}$.
Then we choose $q = \frac{2nt}{nt + 2t - 2n - \eps t}$ so that $\frac 1 p = \frac 1 t + \frac 1 q$.
Since $t \in \pb{\frac n 2, n}$, then $q \in \pr{2, \frac{2n}{n-2}}$ if $\eps < \eps_0 := 2\pr{2 - \frac n t}$.
With this choice of $p$ and $q$, we see from the definition in Theorem \ref{Carlpq} that $\be\pr{2} = \frac {3n+2} {4n} + \frac{3n-2} {8n}\eps$ and $\be\pr{q} = \pr{\frac {3n+2} {2n}} \pr{\frac{2t- n} {2t} } - \brac{1 + \frac{n-2}{4}\pr{\frac n t - 1} + \frac{n-2}{8} \eps}\frac{\eps}{2n}$.
Now we apply Theorem \ref{Carlpq} (ignoring the gradient term) with $q$ as given and with $q = 2$, then sum them to get
\begin{align*}
&\tau^{\be\pr{2}} \norm{\pr{\log r}^{-1} e^{-\tau \phi\pr{r}}u}_{L^2(r^{-n}dx)}
+ \tau^{\be\pr{q}} \norm{\pr{\log r}^{-1} e^{-\tau \phi\pr{r}}u}_{L^q(r^{-n}dx)} \\
&\le C \norm{\pr{\log r} e^{-\tau \phi\pr{r}} r^2 \LP u}_{L^p(r^{-n} dx)} \\
&\le C \norm{\pr{\log r} e^{-\tau \phi\pr{r}} r^2 \pr{\LP u + Vu}}_{L^p(r^{-n} dx)}
+ C \norm{\pr{\log r} e^{-\tau \phi\pr{r}} r^2 V u}_{L^p(r^{-n} dx)} \\
&\le C \norm{\pr{\log r} e^{-\tau \phi\pr{r}} r^2 \pr{\LP u + Vu}}_{L^p(r^{-n} dx)} \\
&+ C \norm{\pr{\log r}^2 r^{2- \frac n p + \frac n q}}_{L^\iny(dx)} \norm{V}_{L^t(dx)} \norm{\pr{\log r}^{-1} e^{-\tau \phi\pr{r}}u}_{L^q(r^{-n}dx)} ,
\end{align*}
where we have applied the triangle inequality along with H\"older's estimate.
Since $\frac 1 p - \frac 1 q = \frac{1}{t}$, then $2- \frac n p + \frac n q = 2 - \frac n t > 0$, and it follows that
\begin{align*}
&\tau^{\be\pr{2}} \norm{\pr{\log r}^{-1} e^{-\tau \phi\pr{r}}u}_{L^2(r^{-n}dx)}
+ \tau^{\be\pr{q}} \norm{\pr{\log r}^{-1} e^{-\tau \phi\pr{r}}u}_{L^q(r^{-n}dx)} \\
&\le C \norm{\pr{\log r} e^{-\tau \phi\pr{r}} r^2 \pr{\LP u + Vu}}_{L^p(r^{-n} dx)} 
+ C M \norm{\pr{\log r}^{-1} e^{-\tau \phi\pr{r}}u}_{L^q(r^{-n}dx)} .
\end{align*}
To absorb the last term into the lefthand side, we need to ensure that $\tau \ge \pr{C M}^{\frac 1 {\be\pr{q}}}$.
Since $\mu = \frac 1 {\be\pr{q}}$, the conclusion follows.
\end{proof}

Now we consider more general elliptic operators with a drift term.
The only difference between the following theorem and \cite[Theorem 7]{DZ17} is in the definition of $\mu$ for smaller values of $t$.
When $t < s$, the value for $\mu$ that we achieve here is better than the one in \cite{DZ17}.
However, compared to \cite{DZ17}, there is no improvement for the range of $t$.
Even though we are now working with Theorem \ref{Carlpq}, which improves upon Theorem \ref{CarlpqO} from \cite{DZ17}, because of the presence of a non-trivial $W$, we cannot increase the range of admissible $t$-values.

\begin{theorem}[cf. Theorem 7 in \cite{DZ17}]
Let $s \in \pb{\frac{3n-2}{2}, \iny}$ and $t \in \pb{ n\pr{\frac{3n-2}{5n-2}}, \iny}$.
Assume that for some $K, M \ge 1$, $\norm{W}_{L^s\pr{B_{R_0}}} \le K$ and $\norm{V}_{L^t\pr{B_{R_0}}} \le M$.
Then there exist constants $C_0$, $C_1 = C_1\pr{n, s, t}$, $C_2 = C_2\pr{n, s, t}$, and sufficiently small $R_0 < 1$  such that for any $u\in C^{\infty}_{0}(B_{R_0}(x_0)\setminus \set{x_0})$ and any large positive constant
$$\tau \ge C_1 K^{\kappa} + C_2 M^{\mu},$$
one has
\begin{align}
\tau^{\be} \norm{(\log r)^{-1} e^{-\tau \phi(r)}u}_{L^2(r^{-n}dx)}
&\leq  C_0 \norm{(\log r ) e^{-\tau \phi(r)} r^2\pr{ \LP u + W \cdot \gr u + V u}}_{L^p(r^{-n} dx)} ,
\label{main1}
\end{align}
where
$\disp p = \left\{\begin{array}{ll}
\frac{2s}{s+2} & t \in \pb{\frac{sn}{s+n}, \iny} \medskip \\
\frac{2 n t }{2n - 2t + nt } & t \in \pb{n\pr{\frac{3n-2}{5n-2}}, \frac{sn}{s+n}}
\end{array}\right.$,
$\disp \kappa = \left\{\begin{array}{ll}
\frac{4s}{2s - \pr{3n-2}} & t \in \pb{\frac{sn}{s+n}, \iny} \medskip \\
\frac{4t}{\pr{5 - \frac 2 n}t - \pr{3n-2}} & t \in \pb{n\pr{\frac{3n-2}{5n-2}}, \frac{sn}{s+n}}
\end{array}\right.$,
$\disp \mu = \left\{\begin{array}{ll}
\frac{4s}{6s - \pr{3n-2}} & t \in \brac{s, \iny} \medskip \\
\frac{4 s t}{6 s t + \pr{n+2}t -4ns + n\pr{n-2}\pr{1 - \frac t s}} & t \in \pr{\frac{sn}{s+n}, s} \medskip \\
\frac{2t}{2t -n} & t \in \pb{n\pr{\frac{3n-2}{5n-2}}, \frac{sn}{s+n}}
\end{array}\right.$, 
$\be = \be\pr{2}$ from Theorem \ref{Carlpq}, and $C_0 = 2C$, where $C$ is from Theorem \ref{Carlpq}.
\label{CarlpqVW}
\end{theorem}

\begin{proof}
If we add inequality \eqref{mainCar} from Theorem \ref{Carlpq} with $q = 2$ to the same inequality with $q$ arbitrary, we see that
\begin{align*}
&\tau^{\be\pr{2}} \norm{(\log r)^{-1} e^{-\tau \phi(r)}u}_{L^2(r^{-n}dx)} \\
&+\tau^{\be\pr{q}} \norm{(\log r)^{-1} e^{-\tau \phi(r)}u}_{L^q(r^{-n}dx)} 
+\tau^{\be_1} \norm{(\log r )^{-1} e^{-\tau \phi(r)}r \gr u}_{L^2(r^{-n}dx)} \\
&\le 2 C \norm{(\log r) e^{-\tau \phi(r)} r^2 (\LP u)}_{L^p(r^{-n} dx)} \\
&\le 2 C \norm{(\log r) e^{-\tau \phi(r)} r^2 (\LP u+ W\cdot \gr u + V u)}_{L^p(r^{-n} dx)} \\
&+ 2 C \norm{(\log r) e^{-\tau \phi(r)} r^2 W\cdot \gr u}_{L^p(r^{-n} dx)}
+ 2 C \norm{(\log r) e^{-\tau \phi(r)} r^2 V u}_{L^p(r^{-n} dx)} \\
&\le 2 C \norm{(\log r) e^{-\tau \phi(r)} r^2 (\LP u+ W\cdot \gr u + V u)}_{L^p(r^{-n} dx)} \\
&+ c \norm{V}_{L^{\frac{pq}{q-p}}\pr{B_{R_0}}} \norm{(\log r)^{-1} e^{-\tau \phi(r)} u}_{L^q(r^{-n} dx)}
+ c \norm{W}_{L^{\frac{2p}{2-p}}\pr{B_{R_0}}} \norm{(\log r)^{-1} e^{-\tau \phi(r)} r \gr u}_{L^2(r^{-n} dx)},
\end{align*}
where we have applied the triangle inequality and H\"older's inequality.
We have also used the ranges of $p$ and $q$ to deduce that $\norm{\pr{\log r}^2 r^{1 + \frac n 2 - \frac n p}}_{L^{\iny}\pr{B_{R_0}}} \le c$ and $\norm{\pr{\log r}^2 r^{2 + \frac n q - \frac n p}}_{L^{\iny}\pr{B_{R_0}}} \le c$.

If we choose $p$ and $q$ so that $\frac{pq}{q-p} \le t$ and $\frac{2p}{2-p} \le s$, then we may use H\"older's inequality in combination with the bounds on $V$ and $W$ to bound the last two terms.
A natural choice is to set $p = \frac{2s}{s+2}$ and $q = \frac{2st}{st+2t-2s}$.
It is clear that $s \in \left(\frac{2n}{n+2}, 2\right]$, however, since the Carleman estimate requires that $q \in \brac{2, \frac{2n}{n-2}}$, then this choice only works when $t \in \brac{\frac{sn}{s+n}, s}$.
Therefore, we have to analyze the inequality in cases according to the relationship between $s$ and $t$.
These ranges exactly coincide with those in \cite{DZ17}.

If $t \in \brac{s, \iny}$, then we set $p = \frac{2s}{s+2}$ and $q = 2$.
For this case, $\frac{2p}{2-p} = s$ and $\frac{pq}{q-p} = \frac{2p}{2-p} = s \le t$.
Moreover, $\be_1 = \frac{1}{2}-\frac{3n-2}{4s}$ and $\be = 1 + \be_1$.
The lower bound on $s$ ensures that $\be_1 > 0$.

Next, if $t \in \pr{\frac{sn}{s+n}, s}$, then we set $p = \frac{2s}{s+2}$ and $q = \frac{2st}{st+2t-2s}$ so that $\frac{pq}{q-p} = t$ and $\frac{2p}{2-p} = s$. 
In this case, $\be_1 = \frac{1}{2}-\frac{3n-2}{4s}$ and $\be = \frac 3 2 - \frac{3n-2} {4s} -  \pr{1 - \frac{n-2} {4s}} n\pr{\frac 1 t - \frac 1 s} = \be_1 + 1 -  \pr{1 - \frac{n-2} {4s}} n\pr{\frac 1 t - \frac 1 s}$.
The lower bound on $s$ and $t$ ensures that $\be_1 > 0$ and $\be > 0$.

Finally, if $t \in \pb{n\pr{\frac{3n-2}{5n-2}}, \frac{sn}{s+n}}$, then we choose $p = \frac{2nt}{2n-2t+nt}$ and $q = \frac{2n}{n-2}$ so that $\frac{pq}{q-p} = t$ and $\frac{2p}{2-p} = \frac{nt}{n-t} \le s$.
Here we have $\be_1 = \frac{1}{2}-\frac{3n-2}{4}\pr{\frac 1 t - \frac 1 n}$ and $\be = 1 - \frac{n}{2t} = \be_1 +  \frac{n-2} 4 \pr{\frac 1 t - \frac 1 n }$.
In this case, it is the lower bound on $t$ that ensures that $\be_1 > 0$, from which it follows that $\be > 0$ as well.

In all three cases, we have that
\begin{align*}
&\pr{\tau^{\be} - cM} \norm{(\log r)^{-1} e^{-\tau \phi(r)}u}_{L^q(r^{-n}dx)} 
+\pr{\tau^{\be_1} - cK} \norm{(\log r )^{-1} e^{-\tau \phi(r)}r \gr u}_{L^2(r^{-n}dx)} \\
&+ \tau^{\be\pr{2}} \norm{(\log r)^{-1} e^{-\tau \phi(r)}u}_{L^2(r^{-n}dx)} 
\le 2 C \norm{(\log r) e^{-\tau \phi(r)} r^2 (\LP u+ W\cdot \gr u + V u)}_{L^p(r^{-n} dx)}.
\end{align*}
Choosing $\tau \ge \pr{2 cK}^{\frac 1 {\be_1}}+ \pr{2 c M}^{\frac 1 {\be}}$ leads to the conclusion of the theorem.
\end{proof}

\section{Order of Vanishing Results}
\label{OofVProof}

Here we describe the proofs of Theorems \ref{OofV} and \ref{OofVW}.
Given that the arguments exactly follow those that appeared in \cite{DZ17}, where now the new Carleman estimates described by Theorems \ref{CarlpqV} and \ref{CarlpqVW} are used, we refer the reader to Section 5 of \cite{DZ17} for the full details.

The first step in the proof is to use the Carleman estimate for the operator to establish a three-ball inequality.
The arguments for the proofs of these three-ball inequalities are based on those in \cite{Ken07}.

\begin{lemma}[cf. Lemma 7 in \cite{DZ17}]
Let $0 < r_0< r_1< R_1 < R_0$, where $R_0 < 1$ is sufficiently small.
Let $t \in \pb{ \frac{n}{2}, \iny}$.
Assume that for some $M \ge 1$, $\norm{V}_{L^t\pr{B_{R_0}}} \le M$.
Let $u : B_{R_{0}} \to \C$ be a solution to \eqref{goal} in $B_{R_0}$.
Then for any $\eps \in \pr{0, \eps_0}$, there exists a constant $C = C\pr{n, t, \eps}$ such that
\begin{align}
\norm{u}_{L^\infty \pr{B_{3r_1/4}}}
&\le C |\log r_1| F\pr{r_1}^{\frac n 2} \pr{|\log r_0| F\pr{r_0} \norm{u}_{L^\iny(B_{2r_0})}}^{k_0} \pr{ |\log R_1|F\pr{R_1} \norm{u}_{L^\iny(B_{R_1})}}^{1 - k_0} \nonumber  \\
&+C F\pr{r_1}^{\frac n 2}  \pr{\frac{R_1}{r_1}}^{\frac n
2}\frac{\abs{\log r_0}}{\abs{\log R_1}} e^{C_1 M^\mu \pr{\phi\pr{\frac{R_1}{2}}-\phi(r_0)}} \norm{u}_{L^\iny(B_{2r_0})},
\label{threeV}
\end{align}
where $\disp k_0 = \frac{\phi(\frac{R_1}{2})-\phi(r_1)}{\phi(\frac{R_1}{2})-\phi(r_0)}$, $F\pr{r} = 1 + r M^{\frac{t}{2t-n}}$, and $\mu$ and $C_1$ are as in Theorem \ref{CarlpqV}.
\end{lemma}

\begin{proof}
We follow the proof of \cite[Lemma 7]{DZ17} with Theorem \ref{CarlpqV} used in place of \cite[Theorem 8]{DZ17}.
\end{proof}

In the setting where we have a drift term, the three-ball inequality takes the following form.

\begin{lemma}[cf. Lemma 6 in \cite{DZ17}]
Let $0 < r_0< r_1< R_1 < R_0$, where $R_0 < 1$ is sufficiently small.
Let $s \in \pb{ \frac{3n-2}{2}, \iny}$, $t \in \pb{ n\pr{\frac{3n-2}{5n-2}}, \iny}$.
Assume that for some $K, M \ge 1$, $\norm{W}_{L^s\pr{B_{R_0}}} \le K$ and $\norm{V}_{L^t\pr{B_{R_0}}} \le M$.
Let $u : B_{R_0} \to \C$ be a solution to $\LP u + W \cdot \gr u + V u = 0$ in $B_{R_0}$.
Then there exists a constant $C = C\pr{n, t, s}$ such that
\begin{align}
\norm{u}_{L^\infty \pr{B_{3r_1/4}}} 
&\le C F\pr{r_1}^{\frac n 2}  |\log r_1| \brac{ (K+|\log r_0|)F\pr{r_0} \norm{u}_{L^\iny(B_{2r_0})}}^{k_0} \nonumber \\
&\times \brac{(K+|\log R_1|)  F\pr{R_1}\norm{u}_{L^\iny(B_{R_1})}}^{1 - k_0} \nonumber \\
&+C F\pr{r_1}^{\frac n 2} \pr{\frac{R_1 }{r_1}}^{\frac n 2} \pr{1 +\frac{|\log r_0|}{K}} \nonumber \\
&\times \exp\brac{\pr{C_1 K^\kappa + C_2 M^\mu} \pr{\phi\pr{\frac{R_1}{2}}-\phi(r_0)}} \norm{u}_{L^\iny(B_{2r_0})},
\label{threeVW}
\end{align}
where $\disp k_0 = \frac{\phi(\frac{R_1}{2})-\phi(r_1)}{\phi(\frac{R_1}{2})-\phi(r_0)}$, $F\pr{r} = 1 + r K^{\frac{s}{s-n}} + r M^{\frac{t}{2t-n}}$, and $\kappa$, $\mu$, $C_1$, $C_2$ are from Theorem \ref{CarlpqVW}.
\label{threeBalls}
\end{lemma}

\begin{proof}
We follow the proof of \cite[Lemma 6]{DZ17} with Theorem \ref{CarlpqVW} used in place of \cite[Theorem 7]{DZ17}.
\end{proof}

To prove Theorems \ref{OofV} and \ref{OofVW}, we first our three-ball inequalities in a propagation of smallness argument to establish a lower bound for the solution in $B_r$.
Then we use the three-ball inequalities again to establish the order of vanishing estimate.

\begin{proof} [Proof of Theorem  \ref{OofV}] 
We follow the proof of \cite[Theorem 3]{DZ17} with estimate \eqref{threeV} used in place of \cite[estimate (5.10)]{DZ17}.
\end{proof}

\begin{proof} [Proof of Theorem  \ref{OofVW}] 
We follow the proof of \cite[Theorem 1]{DZ17} with estimate \eqref{threeVW} used in place of \cite[estimate (5.1)]{DZ17}.
\end{proof}

\section{Proof of the $L^p - L^q$ Carleman estimate}
\label{CarlProofs}

In this section, we prove the new $L^p - L^q$ Carleman estimate that is stated in Theorem \ref{Carlpq}.
To do this, we first prove Proposition \ref{CarL+L-pq}, a different $L^p - L^q$ Carleman.
Then we interpolate the endpoint result of Proposition \ref{CarL+L-pq} (i.e. the $L^p -L^{\frac{2n}{n-2}}$ version) with the $L^p - L^2$ Carleman estimate from \cite{DZ17} presented above in Theorem \ref{CarlpqO}.
Therefore, the essence of the proof is given in Proposition \ref{CarL+L-pq}.

To prove Proposition \ref{CarL+L-pq}, we first use polar coordinates to rewrite the Laplacian as a product of two first order operators.
Next, we project onto the eigenspaces and solve the resulting second order ODE.
Then we employ the eigenfunction estimates of Sogge from \cite{Sog86} to estimate the resulting series.
The crucial ingredients in this step are $L^p - L^q$ estimates for series of eigenfunctions as presented in Lemma \ref{upDown}.
The rough outline of the proof follows the proof from \cite{DZ17}, based on the work in \cite{Reg99}, which uses many of the ideas from \cite{Jer86}.

There are a number of differences between Proposition \ref{CarL+L-pq} and the proof in \cite{DZ17} that make it novel.
In \cite{DZ17}, we decompose the Laplacian into a product of first order operators, then prove Carleman estimates for each operator.
Here, we use that the Laplacian can be written as a product of two first order operators to generate a second order ODE, then we solve that ODE and estimate the operator norm directly.
Therefore, instead of iterating the Carleman estimates, we iterate applications of Young's inequality for convolution.
The techniques from \cite{DZ17} established $L^p - L^2$ Carleman estimates with a gradient term, then employed Sobolev inequalities to extend $q$ to $\frac{2n}{n-2}$.
However, given the obstruction to Carleman estimates with gradient (see for example \cite{Jer86} and our discussion in \cite{DZ17}), this technique didn't allow us to treat values of $t$ near $\frac n2$.
By avoiding the use of a Sobolev inequality here, and directly proving an $L^p - L^q$ estimate, we get a stronger result.
Consequently, we may now consider all $t > \frac n 2$.

We introduce polar coordinates in $\R^n\backslash \{0\}$ by setting $x=r\omega$, with $r=|x|$ and $\omega=(\omega_1,\cdots,\omega_n)\in S$, where we'll use the notation $S = S^{n-1}$.
With $t=\log r$,
$$ \frac{\partial }{\partial x_j}=e^{-t}(\omega_j\partial_t+  \Omega_j), \quad 1\le j\le n, $$
where $\Omega_j$ are vector fields in $S$ that satisfy
$$ \sum^{n}_{j=1}\omega_j\Omega_j=0 \quad \mbox{and} \quad \sum^{n}_{j=1}\Omega_j\omega_j=n-1.$$
In the new coordinate system, the Laplace operator takes the form
\begin{equation}
e^{2t} \LP=\partial^2_t +(n-2)\partial_t+\LP_{\omega},
\label{laplace}
\end{equation}
where $\disp \LP_\omega=\sum_{j=1}^n \Omega^2_j$ is the Laplace-Beltrami operator on $S$.
The eigenvalues for $-\LP_\omega$ are $k(k+n-2)$, where $k\in \mathbb{N}$.
The corresponding eigenspace is $E_k$, the space of spherical harmonics of degree $k$.
We use the notation $\norm{\cdot}_{L^2(dtd\omega)}$ to denote the $L^2$ norm on $(-\iny, \iny)\times S$.
 Let
$$\Lambda=\sqrt{\frac{(n-2)^2}{4}-\LP_\omega}.$$ The operator $\Lambda$ is a
first-order elliptic pseudodifferential operator on $L^2(S)$.
The eigenvalues for the operator $\Lambda$ are $k+\frac{n-2}{2}$, with corresponding eigenspace $E_k$.
That is, for any $v\in C^\iny_0(S)$,
\begin{equation}
\pr{\Lambda - \frac{n-2}{2}} v= \sum_{k\geq 0}k P_k v,
\label{ord}
\end{equation}
where $P_k$ is the projection operator from $L^2(S)$ onto $E_k$.
We remark that the projection operator, $P_k$, acts only on the angular variables.
In particular, $P_k v\pr{t, \om} = P_k v\pr{t, \cdot} \pr{\om}$.
Now define
\begin{equation} L^\pm=\partial_t+\frac{n-2}{2}\pm \Lambda.
\label{use}
\end{equation} From the equation (\ref{laplace}), it follows that
\begin{equation*}
e^{2t} \LP=L^+L^-=L^-L^+.
\end{equation*}

Recall that in the Carleman estimate given by Theorem \ref{Carlpq}, we use the radial weight function
$$\phi(r)=\log r + \log \pr{\log r}^2.$$
With $r=e^t$, define the weight function in terms of $t$ to be
$$\varphi(t)=\phi(e^t)=t + \log t^2.$$
We are only interested in those values of $r$ that are sufficiently small.
Since $r\to 0$ if and only if $t\to-\iny$ then, in terms of the new coordinate $t$, we study the case when $t$ is sufficiently close to $-\iny$.

The following lemma is an important tool and relies on the eigenfunction estimates of Sogge \cite{Sog86}.

\begin{lemma}
Let $\set{c_k}$ be a sequence of numbers.
For any $\frac{2n}{n+2} \le p \le 2 \le q \le \frac{2n}{n-2}$, there exists a constant $C = C \pr{n,p, q}$ such that
\begin{align}
\norm{ \sum c_k P_k v}_{L^{q}(S)}
&\le C \norm{ \innp{c_k}}_{\ell^\iny}^{n^2\pr{\frac 1 q - \frac{n-2}{2n}}\pr{\frac{n+2}{2n} - \frac 1 p}} \pr{\sum \abs{c_k} k^{\frac{n-2}{n}} }^{n^2\pr{\frac 1 2 - \frac 1 q}\pr{\frac 1 p - \frac 1 2}} \nonumber \\
&\times \pr{\sum  c_k^2 k^{\frac{n-2}{n}}}^{ \frac{n^2} 2\brac{\pr{\frac 1 2 - \frac 1 q}\pr{\frac{n+2}{2n} - \frac 1 p}+ \pr{\frac 1 q - \frac{n-2}{2n}}\pr{\frac 1 p - \frac 1 2}}}  
\norm{v}_{L^{p}(S)}.
\label{pqSeries}
\end{align}
\label{upDown}
\end{lemma}

\begin{proof}
Recall that $P_k v = v_k$ is the projection of $v$ onto the space of spherical harmonics of degree $k$.
Sogge's \cite{Sog86} eigenfunction estimates state that there is a constant $C$, depending only on $n \ge 3$, such that for any $v\in L^2(S)$,
\begin{equation}
\norm{ P_k v}_{L^{\frac{2n}{n-2}}(S)}
\le Ck^{\frac{n-2}{n}} \norm{v}_{L^{\frac{2n}{n+2}}(S)}.
\label{sogg}
\end{equation}
By orthogonality, H\"older's inequality, and \eqref{sogg},
\begin{align}
\norm{ \sum c_k P_k v}^2_{L^{2}(S)}
&= \sum \norm{ c_k P_k v}^2_{L^{2}(S)}
\le \sum  c_k^2 \norm{ P_k v}_{L^{\frac{2n}{n-2}}(S)} \norm{v}_{L^{\frac{2n}{n+2}}(S)} \nonumber \\
&\le C \sum  c_k^2 k^{\frac{n-2}{n}} \norm{v}_{L^{\frac{2n}{n+2}}(S)}^2 .
\label{upS}
\end{align}
It is clear that
\begin{align}
\norm{ \sum c_k P_k v}^2_{L^{2}(S)}
&= \sum \norm{ c_k P_k v}^2_{L^{2}(S)}
= \sum  c_k^2 \norm{ P_k v}^2_{L^{2}(S)} \nonumber \\
&\le  \norm{ \innp{c_k}}^2_{\ell^\iny} \norm{v}^2_{L^{2}(S)}
\label{stayS}
\end{align}
and
\begin{align}
\norm{ \sum c_k P_k v}_{L^{\frac{2n}{n-2}}(S)}
&\le \sum \abs{c_k} \norm{ P_k v}_{L^{\frac{2n}{n-2}}(S)}
\le C \sum \abs{c_k} k^{\frac{n-2}{n}} \norm{v}_{L^{\frac{2n}{n+2}}(S)}.
\label{bigUp}
\end{align}
Interpolating \eqref{upS} and \eqref{stayS} shows that for all $\frac{2n}{n+2}\le p\le 2$, there exists $C = C \pr{n,p}$ such that
\begin{align}
\norm{\sum c_k P_k v}_{L^2(S)}
&\le C \norm{ \innp{c_k}}_{\ell^\iny}^{n\pr{\frac{n+2}{2n} - \frac{1}{p}}} \pr{\sum  c_k^2 k^{\frac{n-2}{n}}}^{\frac {n}{2}\pr{\frac 1 p - \frac 1 2}}  \norm{v}_{L^{p}(S)}.
\label{upp}
\end{align}
Using this estimate, we have
\begin{align*}
\innp{u, \sum c_k P_k v}
&= \innp{\sum c_k P_k  u, v}
\le \norm{\sum c_k P_k u}_{L^2\pr{S}} \norm{v}_{L^2\pr{S}} \\
&\le C \norm{ \innp{c_k}}_{\ell^\iny}^{n\pr{\frac{n+2}{2n} - \frac{1}{p}}} \pr{\sum  c_k^2 k^{\frac{n-2}{n}}}^{\frac {n}{2}\pr{\frac 1 p - \frac 1 2}} \norm{u}_{L^{p}(S)} \norm{v}_{L^{2}(S)}.
\end{align*}
By duality, we conclude that for all $2 \le q \le \frac{2n}{n-2}$, there exists $C = C \pr{n,q}$ such that
\begin{align}
\norm{\sum c_k P_k v}_{L^q(S)} 
&\le C \norm{ \innp{c_k}}_{\ell^\iny}^{n\pr{\frac 1 q - \frac{n-2}{2n}}} \pr{\sum  c_k^2 k^{\frac{n-2}{n}}}^{\frac n 2\pr{\frac 1 2 - \frac 1 q}} \norm{v}_{L^{2}(S)}.
\label{upq}
\end{align}
Now if we interpolate \eqref{upS} and \eqref{bigUp}, we see that for all $2 \le q \le \frac{2n}{n-2}$, there exists $C = C \pr{n,q}$ such that
\begin{align}
\norm{ \sum c_k P_k v}_{L^{q}(S)}
&\le C \pr{\sum  c_k^2 k^{\frac{n-2}{n}}}^{\frac n 2\pr{\frac 1 q - \frac{n-2}{2n}}} \pr{\sum \abs{c_k} k^{\frac{n-2}{n}} }^{n\pr{\frac 1 2 - \frac 1 q}}\norm{v}_{L^{\frac{2n}{n+2}}(S)}.
\label{up2q}
\end{align}
Finally, we interpolate \eqref{upq} and \eqref{up2q} to reach the conclusion of the lemma.
\end{proof}

While the very general form of Lemma \ref{upDown} will serve us in the proof of our proposition below, its simplest form is also quite useful.

\begin{corollary}
For any $\frac{2n}{n+2} \le p \le 2 \le q \le \frac{2n}{n-2}$, there exists a constant $C = C \pr{n,p, q}$ such that
\begin{align}
\norm{ P_k v}_{L^{q}(S)}
&\le C  k^{ \frac{n-2} 2 \pr{\frac 1 p - \frac 1 q}}
\norm{v}_{L^{p}(S)}.
\label{pqEst}
\end{align}
\label{upDownCor}
\end{corollary}

Now we prove an $L^p-L^q$ type Carleman estimate for the operator $\LP$.

\begin{proposition}
Let $\frac{2n}{n+2} \le p \le 2 \le q \le \frac{2n}{n-2}$ be such that $\frac 1 p - \frac 1 q < \frac 2 n$.
There exists a constant $C = C\pr{n,p,q}$ and a sufficiently small $R_0 \le 1$ such that for any $v\in C^{\iny}_{0}\pr{B_{R_0}(x_0)\backslash\set{x_0} }$ and any $\tau>1$, one has
\begin{equation}
\norm{ \pr{\log r}^{-1} e^{-\tau \phi\pr{r}}v}_{L^q(r^{-n}dx)}
\le  C \tau^{-\be_2} \norm{(\log r ) e^{-\tau \phi\pr{r}} r^2 \LP v}_{L^p(r^{-n} dx)},
\label{key}
\end{equation}
where $\be_2 = 1 - \frac{n}{2}\pr{\frac 1 p - \frac 1 q}$. 
\label{CarL+L-pq}
\end{proposition}

\begin{proof}
Recalling the definitions of $t$, $\vp$, $L^\pm$ from above, proving \eqref{key} is equivalent to showing that
\begin{equation}
\norm{t^{-1} e^{-\tau \vp\pr{t}} v}_{L^q(dtd\omega)} 
\le C \tau^{-\be_2} \norm{ t e^{-\tau \vp\pr{t}} L^+L^- v}_{L^p(dtd\omega)}
\label{key-}
\end{equation}
for all $v \in C^\iny_0\pr{\pr{-\iny, t_0} \times S}$, where $\abs{t_0} = - t_0$ is sufficiently large.
To prove \eqref{key-}, we introduce the conjugated operators $L^\pm_\tau$ of $L^\pm$, defined by
$$L^\pm_\tau u = e^{-\tau \vp\pr{t}}L^\pm(e^{\tau \vp\pr{t}}u) = L^\pm + \tau \vp'\pr{t},$$
where we have used definition \eqref{use}.
Notice that
\begin{align*}
L^+_\tau L^-_\tau { u}
= L^+_\tau \pr{L^-_\tau u}
= L^+_\tau \pr{e^{-\tau \vp\pr{t}}L^-(e^{\tau \vp\pr{t}}u)}
=e^{-\tau \vp\pr{t}} L^+ L^-(e^{\tau \vp\pr{t}}u)
= \pr{L^+L^-}_\tau u.
\end{align*}
With $v=e^{\tau \vp\pr{t}}u$, inequality \eqref{key-} is equivalent to
\begin{equation}
\norm{t^{-1} u}_{L^q(dtd\omega)}\le C \tau^{-\be_2} \norm{t L^+_\tau L^-_\tau u}_{L^p(dtd\omega)}. 
\label{keyu}
\end{equation}
From \eqref{laplace}, \eqref{ord} and \eqref{use}, the operator 
$L^+_\tau L^-_\tau$ takes the form
\begin{align*}
L^+_\tau L^-_\tau &= \sum_{k\geq 0} \pr{\partial_t+\tau \vp'\pr{t} + k + n -2} \pr{\partial_t+\tau \vp'\pr{t} -k}P_k
\end{align*}
which leads to the second order differential equation
\begin{align*}
P_k L^+_\tau L^-_\tau u &=  \pr{\partial_t+ \tau \vp'\pr{t} + k + n -2} \pr{\partial_t+\tau \vp'\pr{t} -k} P_k u.
\end{align*}
For $u\in C^\iny_{0}\pr{ (-\iny, \ t_0)\times S}$, solving this ODE gives
\begin{align}
P_k u(t, \omega)
&=- \int_{t}^{\iny} \int_{- \iny}^y e^{k\pr{t-y} - \tau \pr{\vp\pr{t} - \vp\pr{y}}}  e^{-\pr{k+n-2}\pr{y-s} - \tau \pr{ \vp\pr{y} - \vp\pr{s}}} P_k L^+_\tau L^-_\tau u (s, \omega)\, ds \, dy
\nonumber \\
&= \int_{-\iny}^t  \int_{- \iny}^y e^{ k\pr{t-y} - \tau \pr{\vp\pr{t} - \vp\pr{y}}}  e^{-\pr{k+n-2}\pr{y-s} - \tau \pr{ \vp\pr{y} - \vp\pr{s}}} P_k L^+_\tau L^-_\tau u (s, \omega)\, ds \, dy.
\label{starr}
\end{align}
Both of these representations will be used below in our estimates.

Since $\disp u = \sum_{k\geq 0} P_k u$, we split the sum into two parts.
Let $M=\lceil 2 \tau\rceil$ and define $\disp P^+=\sum_{k\ge M}P_k$ and $\disp P^-=\sum_{k=0}^{M-1}P_k$.
In order to prove the \eqref{keyu}, it suffices to show that
\begin{equation}
\norm{ t^{-1} P^+ {u} }_{L^q(dtd\om)} 
\le C \tau^{-\be_2}\norm{t L_\tau^+ L_\tau^- u}_{L^p(dtd\omega)}
\label{key1}
\end{equation}
and
\begin{equation}
\norm{t^{-1} P^- u}_{L^q(dtd\om)}
\le C \tau^{-\be_2}\norm{t L_\tau^+ L_\tau^- u}_{L^p(dtd\omega)}
\label{key2}
\end{equation}
for all $u \in C^\iny_0\pr{(-\iny, \ t_0)\times S}$. 
The sum of \eqref{key1} and \eqref{key2} will yield \eqref{keyu}, which implies \eqref{key-}, proving the lemma.
We first establish \eqref{key1}.

For $k\geq M \ge 2 \tau$, we'll used the first line in \eqref{starr}.
Recalling that $\vp\pr{t} = t + \log t^2$, we see that if $\abs{t_0} \ge 4$, then
\begin{align}
\chi_{\set{y \ge t}} e^{ - \tau \pr{\vp\pr{t} - \vp\pr{y}}}
&= \chi_{\set{y \ge t}} e^{ \tau \abs{y-t} + 2 \tau \log\pr{1 - \abs{\frac{ y-t} t}}} 
\le e^{\tau \abs{y-t}} 
\nonumber \\
\chi_{\set{s \le y}} e^{ - \tau \pr{\vp\pr{y} - \vp\pr{s}}}
&= \chi_{\set{s \le y}} e^{ -\tau \abs{s-y} - 2\tau \log\pr{1 - \abs{\frac{y-s} s}}}
\le e^{ -\tau \abs{s-y} + 2\tau \abs{\frac{y-s} s}} 
\le e^{ -\frac \tau 2 \abs{s-y}} 
\label{evpBds}
\end{align}
and it follows that
\begin{align*}
\chi_{\set{y \ge t}} \chi_{\set{s \le y}} e^{ k\pr{t-y} - \tau \pr{\vp\pr{t} - \vp\pr{y}}}  e^{-\pr{k+n-2}\pr{y-s} - \tau \pr{ \vp\pr{y} - \vp\pr{s}}} 
&\le e^{-\frac k 2\abs{t-y}} e^{- k \abs{y-s}}.
\end{align*}
Taking the $L^q\pr{S}$-norm in \eqref{starr} and using this bound gives that
\begin{equation*}
\norm{P_k u(t, \cdot)}_{L^q(S)}
\le \int_{-\iny}^{\iny} \int_{-\iny}^{\iny} e^{-\frac k 2\abs{t-y}} e^{- k \abs{y-s}} \norm{P_k L^+_\tau L^-_\tau u(s, \cdot)}_{L^q(S)} \,ds \, dy.
\end{equation*}
With the aid of \eqref{pqEst}, we get
\begin{equation*}
\norm{P_k u(t, \cdot)}_{L^q(S)}
\le C k^{\frac{n-2}{2}\pr{\frac 1 p - \frac 1 q}} \int_{-\iny}^{\iny} e^{-\frac k 2 \abs{t-y}} \int_{-\iny}^{\iny} e^{-k \abs{y-s}} \norm{L^+_\tau L^-_\tau u(s,\cdot)}_{L^p(S)} \,ds \, dy.
\end{equation*}
Consecutively applying Young's inequality for convolution then yields
\begin{align*}
\norm{P_k u}_{L^q(dt d\omega)}
&\le C k^{\frac{n-2}{2}\pr{\frac 1 p - \frac 1 q}} \pr{\int_{-\iny}^{\iny} e^{-\frac{\si}{2} k|z|} dz}^{\frac{1}{\sigma}} \pr{\int_{-\iny}^{\iny} e^{-\rho k |z|} dz}^{\frac{1}{\rho}} \norm{L^+_\tau L^-_\tau u}_{L^p(dt d\om)}
\end{align*}
with $\frac{1}{\sigma}=1 + \frac 1 q -\frac{1}{a}$ and $\frac 1 \rho = 1 + \frac 1 a - \frac 1 p$ for any index $a$ in the appropriate range.
Since
$$ \pr{\int_{-\iny}^{\iny} e^{- \frac \si 2 k|z|}}^{\frac{1}{\sigma}}\le C k^{\frac{1}{a} - 1 - \frac 1 q} \quad \text{ and } \quad
\pr{\int_{-\iny}^{\iny} e^{-\rho k |z|} dz}^{\frac{1}{\rho}} \le C k^{\frac 1 p - \frac 1 a - 1}$$
then
\begin{equation}
\norm{P_k u}_{L^q(dt d\omega)}
\le C k^{ - 2 + \frac{n}{2}\pr{\frac 1 p - \frac 1 q}} \norm{L^+_\tau L^-_\tau u}_{L^p(dt d\om)}.
\label{highTerms}
\end{equation}
Summing up $k\ge M$ shows that
\begin{align*}
\norm{P^+ u}_{L^q(dt d\omega)}
&\le \sum_{k\ge M} \norm{P_k u}_{L^q(dt d\omega)}
\le C  \sum_{k\ge M}  k^{-2 + \frac{n}{2}\pr{\frac 1 p - \frac 1 q}} \norm{L^+_\tau L^-_\tau u}_{L^p(dt d\om)} \\ 
&\le C \tau^{-1 + \frac{n}{2}\pr{\frac 1 p - \frac 1 q}}  \norm{L^+_\tau L^-_\tau u}_{L^p(dt d\om)},
\end{align*}
where we have used the condition that $\frac 1 p - \frac 1 q < \frac 2 n$ to conclude that the series converges.
Since $\abs{t_0}$ is assumed to be sufficiently large, \eqref{key1} follows from this bound and the definition of $\be_2$.

Now fix $t \in \pr{-\iny, - \abs{t_0}}$ and set $N=\lceil\tau \vp'\pr{t} \rceil$.
We'll estimate $\norm{ P^- u}_{L^2(dtd\omega)}$ by summing in two parts; from $N$ to $M-1$ and then from $0$ to $N-1$.
That is, we'll sum the parts where $k > N$ and $k < N$ separately using each of the representations from \eqref{starr}.
An application Taylor's theorem shows that for all $t, y \in \pr{-\iny, t_0}$
\begin{align*}
\vp\pr{y}-\vp\pr{t}
&= \vp'\pr{t}\pr{y-t} + \frac{1}{2}\vp''\pr{y_0}\pr{y-t}^2
=  \vp'\pr{t}\pr{y-t} - \frac{1}{y_0^2}\pr{y-t}^2,
\end{align*}
where $y_0$ is some number between $y$ and $t$.
This more refined estimate shows that
\begin{align*}
&\chi_{\set{y \ge t}} e^{ - \tau \pr{\vp\pr{t} - \vp\pr{y}}}
\le e^{ \tau  \vp'\pr{t}\abs{y-t} - \frac{\tau}{t^2}\pr{y-t}^2} \\
&\chi_{\set{y \le t}} \chi_{\set{s \le y}} e^{ - \tau \pr{\vp\pr{t} - \vp\pr{y}}}
\le e^{- \tau \vp'\pr{t}\abs{y-t} - \frac{\tau}{s^2}\pr{y-t}^2}.
\end{align*}
By combining these observations with \eqref{evpBds}, we see that
\begin{align}
& \chi_{\set{y \ge t}} \chi_{\set{s \le y}} e^{ k\pr{t-y} - \tau \pr{\vp\pr{t} - \vp\pr{y}}}  e^{-\pr{k+n-2}\pr{y-s} - \tau \pr{ \vp\pr{y} - \vp\pr{s}}} \nonumber \\
&\le e^{ -\pr{k -  \tau \vp'\pr{t}}\abs{y-t} - \frac{\tau}{t^2}\pr{y-t}^2 -\pr{k + \frac \tau 2}\abs{y-s}}
\label{highSum} \\
& \chi_{\set{y \le t}} \chi_{\set{s \le y}} e^{ k\pr{t-y} - \tau \pr{\vp\pr{t} - \vp\pr{y}}}  e^{-\pr{k+n-2}\pr{y-s} - \tau \pr{ \vp\pr{y} - \vp\pr{s}}} \nonumber \\
&\le e^{ - \pr{\tau \vp'\pr{t} - k}\abs{y-t} - \frac{\tau}{s^2}\pr{y-t}^2 -\pr{k + \frac \tau 2}\abs{y-s}} .
\label{lowSum}
\end{align}
From the first line of \eqref{starr}, we sum over $k$ and use \eqref{highSum} to get
\begin{align*}
\sum^{M-1}_{k=N} P_k u(t, \om)
&\le \iint \sum^{M-1}_{k=N}  e^{ -\pr{k -  \tau \vp'\pr{t}}\abs{y-t} - \frac{\tau}{t^2}\pr{y-t}^2 - \tau \abs{y-s}} P_k L^+_\tau L^-_\tau u (s, \om) ds \, dy.
\end{align*}
Then we apply an $L^q\pr{S}$-norm to get
\begin{align}
& \norm{\sum^{M-1}_{k=N} P_k u(t, \cdot)}_{L^q\pr{S}} \nonumber \\
&\le \iint \norm{\sum^{M-1}_{k=N}  e^{ -\pr{k -  \tau \vp'\pr{t}}\abs{y-t} - \frac{\tau}{t^2}\pr{y-t}^2 - \tau \abs{y-s}} P_k L^+_\tau L^-_\tau u (s, \cdot)}_{L^q\pr{S}} ds \, dy.
\label{normMid}
\end{align}
For $0 \le k\le N-1$, we use the second line from \eqref{starr} in combination with \eqref{lowSum} to similarly obtain
\begin{align}
& \norm{\sum^{N-1}_{k=0} P_k u(t, \cdot)}_{L^q\pr{S}} \nonumber \\
&\le  \iint \norm{\sum^{N-1}_{k=0}  e^{ - \pr{\tau \vp'\pr{t} - k}\abs{y-t} - \frac{\tau}{s^2}\pr{y-t}^2 - \frac \tau 2 \abs{y-s}} P_k L^+_\tau L^-_\tau u (s, \cdot)}_{L^q\pr{S}} ds \, dy.
\label{normlow}
\end{align}
For these inner norms, Lemma \ref{upDown} is applicable with $c_k = e^{ -\pr{k -  \tau \vp'\pr{t}}\abs{y-t} - \frac{\tau}{t^2}\pr{y-t}^2 - \tau \abs{y-s}}$ and $c_k = e^{ - \pr{\tau \vp'\pr{t} - k}\abs{y-t} - \frac{\tau}{s^2}\pr{y-t}^2 - \frac \tau 2 \abs{y-s}}$.
Estimate \eqref{pqSeries} shows that
\begin{align*}
&\norm{\sum^{M-1}_{k=N} e^{ -\pr{k -  \tau \vp'\pr{t}}\abs{y-t} - \frac{\tau}{t^2}\pr{y-t}^2 - \tau \abs{y-s}} P_k L^+_\tau L^-_\tau u (s, \cdot)}_{L^q\pr{S}} \\
&\le C \tau^{\frac{n-2}{2} \pr{\frac 1 p - \frac 1 q}}  
e^{- \tau \pr{\abs{y-s} + \frac{\pr{y-t}^2}{t^2}}} 
\pr{\sum^{M-1}_{k=N} {e^{ -\pr{k -  \tau \vp'\pr{t}}\abs{y-t}}} }^{n^2\pr{\frac 1 2 - \frac 1 q}\pr{\frac 1 p - \frac 1 2}} \nonumber \\
&\times \pr{\sum^{M-1}_{k=N}  e^{ -2\pr{k -  \tau \vp'\pr{t}}\abs{y-t} } }^{ \frac{n^2} 2\brac{\pr{\frac 1 2 - \frac 1 q}\pr{\frac{n+2}{2n} - \frac 1 p}+ \pr{\frac 1 q - \frac{n-2}{2n}}\pr{\frac 1 p - \frac 1 2}}}  
\norm{L^+_\tau L^-_\tau u (s, \cdot)}_{L^p\pr{S}} \\
&\le C \tau^{\al_1}  
\pr{\frac{1}{\abs{y-t} } }^{\al_2} 
e^{- \tau \pr{\abs{y-s} + \frac{\pr{y-t}^2}{t^2}} } 
\norm{L^+_\tau L^-_\tau u (s, \cdot)}_{L^p\pr{S}} ,
\end{align*}
where $\al_1 = \frac{n-2}{2} \pr{\frac 1 p - \frac 1 q}$, $\al_2 =  \frac{n} 2 \pr{\frac 1 p - \frac 1 q}$.
Similarly, \eqref{pqSeries} shows that
\begin{align*}
&\norm{\sum^{N-1}_{k=0}  e^{ - \pr{\tau \vp'\pr{t} - k}\abs{y-t} - \frac{\tau}{s^2}\pr{y-t}^2 - \frac \tau 2 \abs{y-s}} P_k L^+_\tau L^-_\tau u (s, \cdot)}_{L^q\pr{S}} \\
&\le C \tau^{\al_1}  
\pr{\frac{1}{\abs{y-t} } }^{\al_2} 
e^{- \tau \pr{\frac{\abs{y-s}}2 + \frac{\pr{y-t}^2}{s^2}} } 
\norm{L^+_\tau L^-_\tau u (s, \cdot)}_{L^p\pr{S}}.
\end{align*}
We see that
\begin{align*}
e^{- \tau \frac{\pr{y-t}^2}{t^2} } 
\le C_j \pr{1+\frac{\tau}{t^2}\pr{y-t}^2}^{-j}  
\end{align*}
for all $j\geq 0$ so that with $j=\frac{1}{2}$, we have
\begin{equation*}
e^{- \tau \frac{\pr{y-t}^2}{t^2} } 
\le C\abs{t} \pr{1+\sqrt{\tau}\abs{y-t}}^{-1}.
\end{equation*}
Similarly,
\begin{equation*}
e^{- \tau \frac{\pr{y-t}^2}{s^2} } 
\le C\abs{s} \pr{1+\sqrt{\tau}\abs{y-t}}^{-1}.
\end{equation*}
Combining \eqref{normMid} and \eqref{normlow} with the bounds produced above shows that
\begin{align*}
\norm{\sum^{M-1}_{k=N} P_k u(t, \cdot)}_{L^q\pr{S}}
&\le C \tau^{\al_1} \iint e^{- \tau \abs{y-s} } 
\frac{\abs{t}}{ \abs{y-t}^{\al_2}\pr{1+\sqrt{\tau}\abs{y-t}}}
\norm{L^+_\tau L^-_\tau u (s, \cdot)}_{L^p\pr{S}} ds \, dy
\end{align*}
and
\begin{align*}
\norm{\sum^{N-1}_{k=0} P_k u(t, \cdot)}_{L^q\pr{S}} 
&\le C \tau^{\al_1} \iint e^{-\frac{\tau}{2}  \abs{y-s}} 
\frac{\abs{s}}{\abs{y-t}^{\al_2}\pr{1+\sqrt{\tau}\abs{y-t}}}
\norm{L^+_\tau L^-_\tau u (s, \cdot)}_{L^p\pr{S}} ds \, dy.
\end{align*}
Now we sum and use the fact that $s, t \in \pr{-\iny, t_0}$ where $\abs{t_0}$ is sufficiently large to conclude that
\begin{align*}
\norm{{t}^{-1} P^- u(t, \cdot)}_{L^q\pr{S}} 
&\le C \tau^{\al_1 + \frac{\al_2}2} \int \frac{\pr{\sqrt \tau \abs{y-t}}^{-\al_2} }{\pr{1+\sqrt{\tau}\abs{y-t}}} \int  e^{-\frac{\tau}{2}  \abs{y-s}}
\norm{\abs{s}L^+_\tau L^-_\tau u (s, \cdot)}_{L^p\pr{S}} ds \, dy.
\end{align*}
Repeated applications of Young's inequality for convolution shows that
\begin{align*}
\norm{{t}^{-1} P^- u}_{L^q\pr{dt d\om}} 
&\le C \tau^{\al_1 + \frac{\al_2}{2}}  
\pr{\int \frac{\pr{\sqrt \tau \abs{\zeta}}^{-\si \al_2}}{\pr{1+\sqrt{\tau}\abs{\zeta}}^\si} d\zeta }^{\frac 1 \si}
\pr{\int e^{-\frac{\rho}{2} \tau \abs{z}}dz }^{\frac 1 \rho}
\norm{t L^+_\tau L^-_\tau u}_{L^p\pr{dt d\om}},
\end{align*}
where $\frac 1 \si = 1 + \frac 1 q - \frac 1 a$ and $\frac 1 \rho = 1 + \frac 1 a - \frac 1 p$ for any $a$ in the appropriate range.
We choose $a = q$ and then $\si = 1$.
Since $\al_2 \in \pr{0, 1}$, then
\begin{align*}
\pr{\int \frac{\pr{\sqrt \tau \abs{\zeta}}^{-\si \al_2}}{\pr{1+\sqrt{\tau}\abs{\zeta}}^\si} d\zeta }^{\frac 1 \si}
&= \int \frac{\pr{\sqrt \tau \abs{\zeta}}^{-\al_2}}{\pr{1+\sqrt{\tau}\abs{\zeta}}} d\zeta 
= C \tau^{- \frac 1 2}.
\end{align*}
Moreover, $\disp \pr{\int e^{-\frac{\rho}{2} \tau \abs{z}}dz }^{\frac 1 \rho} = C \tau^{-\frac 1 \rho}$ where $-\frac 1 \rho = \frac 1 p - \frac 1 q - 1$.
Since $\al_1 + \frac{\al_2}{2} - \frac 1 2 - \frac 1 \rho = - \frac 3 2 + \frac{3n}{4} \pr{\frac 1 p - \frac 1 q}$, we conclude that
\begin{align}
\norm{{t}^{-1} P^- u}_{L^q\pr{dt d\om}} 
&\le C \tau^{ - \frac 3 2 + \frac{3n}{4} \pr{\frac 1 p - \frac 1 q}} \norm{t L^+_\tau L^-_\tau u}_{L^p\pr{dt d\om}}.
\label{key22}
\end{align}
Since $\frac 1 p - \frac 1 q < \frac 2 n$ implies that $- \frac 3 2 + \frac{3n}{4} \pr{\frac 1 p - \frac 1 q} \le - 1 + \frac n 2  \pr{\frac 1 p - \frac 1 q} = - \be_2$, then \eqref{key2} follows, completing the proof.
\end{proof}

To prove Theorem \ref{Carlpq}, we interpolate Proposition \ref{CarL+L-pq} at $q = \frac{2n}{n-2}$ with the $q = 2$ version of Theorem \ref{CarlpqO} that originally appeared in \cite{DZ17}.

\begin{proof}[Proof of Theorem \ref{Carlpq}]
Let $u\in C^{\iny}_{0}\pr{B_{R_0}(x_0)\backslash\set{x_0} }$.
Since Theorem \ref{CarlpqO} implies that
\begin{align*}
\tau^{\be_1} \norm{ \pr{\log r}^{-1} e^{-\tau \vp\pr{r}} r \gr u}_{L^2(r^{-n}dx)} 
&\le C \norm{(\log r ) e^{-\tau \vp\pr{r}} r^2 \LP u}_{L^p(r^{-n} dx)},
\end{align*}
then it suffices to show that
\begin{align*}
& \tau^\be \norm{\pr{\log r}^{-1} e^{-\tau \phi\pr{r}}u}_{L^q(r^{-n}dx)}
\le  \norm{(\log r ) e^{-\tau \phi\pr{r}} r^2 \LP u}_{L^p(r^{-n} dx)}.
\end{align*}
For any $q \in \brac{2, \frac{2n}{n-2}}$, we write $q = 2 \te + \frac{2n}{n-2}\pr{1 - \te}$ for some $\te \in \brac{0,1}$.
An application of H\"older's inequality followed by estimate \eqref{mainCarO} with $q = 2$ applied to the first term in the product and estimate \eqref{key} with $q = \frac{2n}{n-2}$ applied to the second term in the product shows that
\begin{align*}
& \norm{\pr{\log r}^{-1} e^{-\tau \phi\pr{r}}u}_{L^q(r^{-n}dx)} \\
\le& \norm{\pr{\log r}^{-1} e^{-\tau \phi\pr{r}}u}_{L^2(r^{-n}dx)}^{\frac{2 \te}{q}} 
\norm{\pr{\log r}^{-1} e^{-\tau \phi\pr{r}}u}_{L^{\frac{2n}{n-2}}(r^{-n}dx)}^{\frac{2n\pr{1-\te}}{\pr{n-2}q}} \\
\le& \pr{C \tau^{-\be_0} \norm{(\log r ) e^{-\tau \phi\pr{r}} r^2 \LP u}_{L^p(r^{-n} dx)} }^{\frac{2 \te}{q}}
\pr{C \tau^{-\be_2} \norm{(\log r ) e^{-\tau \phi\pr{r}} r^2 \LP u}_{L^p(r^{-n} dx)} }^{\frac{2n\pr{1-\te}}{\pr{n-2}q}} \\
=& C \tau^{-\be_0\frac{2 \te}{q}-\be_2{\frac{2n\pr{1-\te}}{\pr{n-2}q}}}  \norm{(\log r ) e^{-\tau \phi\pr{r}} r^2 \LP u}_{L^p(r^{-n} dx)}.
\end{align*}
A computation shows that $\be_0\frac{2 \te}{q} + \be_2{\frac{2n\pr{1-\te}}{\pr{n-2}q}} = \be$, and the conclusion follows.
\end{proof}

%
%
%
%
%
%

\def\cprime{$'$}


\end{document}